\documentclass[12pt,letterpaper]{amsart}
\usepackage[leqno]{amsmath}
\usepackage{amsfonts}
\usepackage{amssymb}
\usepackage{amsthm}
\usepackage{amssymb}
\usepackage{multicol}

\usepackage{calc}

\theoremstyle{plain}
\newtheorem{theorem}{Theorem}[section]
\newtheorem{corollary}[theorem]{Corollary}
\newtheorem{proposition}[theorem]{Proposition}

\theoremstyle{definition}

\begin{document}

\begin{center}
{\Large \bf Dynamic approach to k-forcing}\\[5ex]

\begin{multicols}{2}

Yair Caro\\
{\small University of Haifa-Oranim\\}

\columnbreak

Ryan Pepper\footnote{corresponding author: pepperr@uhd.edu}\\
{\small University of Houston-Downtown\\}

\end{multicols}

\textbf{Abstract}

\end{center}


\noindent 
The k-forcing number of a graph is a generalization of the zero forcing number. In this note, we give a greedy algorithm to approximate the k-forcing number of a graph. Using this dynamic approach, we give corollaries which improve upon two theorems from a recent paper of Amos, Caro, Davila and Pepper \cite{ACDP}, while also answering an open problem posed by Meyer \cite{Meyer}.


\section{Introduction and Key Definitions}
Throughout this paper, all graphs are simple, undirected and finite. Let $G=(V,E)$ be a graph. We will use the basic notation: $n=n(G)=|V|$, $m=m(G)=|E|$, $\Delta(G)$, and $\delta(G)$; to denote respectively the order, size, maximum degree and minimum degree of $G$.  To denote the degree of a vertex $v$, we will write $deg(v)$. Other necessary definitions will be presented throughout the paper as needed, and for basic graph theory definitions, the reader can consult \cite{West}. 

Now we introduce and define the $k$-forcing number.  Let $k \leq \Delta$ be a positive integer.  A set $S \subseteq V$ is a \emph{$k$-forcing set} if, when its vertices are initially colored (state 1) -- while the remaining vertices are intitially non-colored (state 0) -- and the graph is subjected to the following color change rule, all of the vertices in $G$ will eventually become colored (state 1). A colored vertex with at most $k$ non-colored neighbors will cause each non-colored neighbor to become colored.  The \emph{$k$-forcing number}, which we will denote with $F_k=F_k(G)$, is the cardinality of a smallest $k$-forcing set.  We will call the discrete dynamical process of applying this color change rule to $S$ and $G$ the \emph{$k$-forcing process}. If a vertex $v$ causes a vertex $w$ to change colors during the $k$-forcing process, we say that $v$ \emph{$k$-forces} $w$ (and we note here that a vertex can be $k$-forced by more than one other vertex). Our paper is about upper bounds on the $k$-forcing number.

This concept generalizes the recently introduced but heavily studied notion of the zero forcing number of a graph, which is denoted $Z=Z(G)$.  Indeed, $F_1(G)=Z(G)$, and throughout this paper, we will denote the zero forcing number with $F_1(G)$. The zero forcing number was introduced independently in \cite{AIM} and \cite{Burgarth}.  In \cite{AIM}, it is introduced to bound from below the minimum rank of a graph, or equivalently, to bound from above the maximum nullity of a graph.  Namely, if $G$ is a graph whose vertices are labeled from $1$ to $n$, then let $M(G)$ denote the maximum nullity over all symmetric real valued matrices where, for $i \ne j$, the $ij^{th}$ entry is nonzero if and only if $\{i,j\}$ is an edge in $G$. Then, the zero forcing number is an upper bound on $M(G)$, that is, $F_1(G)=Z(G) \ge M(G)$. In \cite{Burgarth}, it is indirectly introduced in relation to a study of control of quantum systems.  Besides its origins in the minimum rank/maximum nullity problem and the control of quantum systems, one can imagine other applications in the spread of opinions or disease in a social network (as described for a similar invariant by Dreyer and Roberts in \cite{Dreyer}). Some of the many other papers written about zero forcing number are: \cite{ACDP, Chilakammari, Edholm, Eroh, Hogben, Meyer, Row2, Row, Yi}. In most of these papers, the tools used to study the zero forcing number (or $k$-forcing number in \cite{ACDP}) were algebraic and non-constructive. In this paper, the main result is a constructive algorithm and the arguments are graph theoretic.
\section{Main result}
Before stating our main theorem, we observe this preliminary result.
\begin{proposition}
Let $k$ be a positive integer. If $G=(V,E)$ is a connected graph with maximum degree $\Delta \leq k$, then $F_k(G)=1$.
\end{proposition}
\begin{proof}
Assume $G$ is a connected graph with $\Delta \leq k$. Let $v$ be a vertex in $G$ and color $v$. Since $deg(v) \leq \Delta \leq k$, $v$ has at most $k$ uncolored neighbors. Begin the $k$-forcing process and $v$ will force each of its neighbors to be colored. Each of those in turn will force their neighbors to be colored, since they also have degree at most $k$. Since $G$ is connected, all the vertices will become colored and $F_k(G)=1$, as claimed. 
\end{proof}

This gives, for example, that the $3$-forcing number of connected cubic graphs is exactly equal to one. Now that the trivial case is dealt with, we move on to the more interesting case that $\Delta \geq k+1$, and present the main result below.

\begin{theorem}\label{main}
Let $k$ be a positive integer and let $G=(V,E)$ be a connected graph with minimum degree $\delta$ and maximum degree $\Delta \geq k+1$.

i) If $\delta < \Delta = k+1$, then $F_k(G) =1$.

ii) If $\delta =\Delta=k+1$, then $F_k(G) =2$.

iii) Otherwise, if $\Delta \geq k+2$, then the following inequality holds, 
\[
F_k(G) \leq \frac{(\Delta-k-1)n(G) + \max\{\delta(k+1-\Delta)+k, k(\delta-\Delta +2)\}}{\Delta -1}.
\]
\end{theorem}

\begin{proof}
First, we prove i). Assume $\delta < \Delta = k+1$. Let $v$ be a vertex in $G$ of degree $\delta$. Color $v$ and observe that since $v$ has at most $k$ non-colored neighbors, all of its neighbors become colored during the first step of the $k$-forcing process. Since each of those vertices now have at least one colored neighbor and degree at most $k+1$, they also have at most $k$ non-colored neighbors. Thus, they each force their neighborhoods to become colored. Since $G$ is connected, this process repeats until all vertices are colored. Hence, $F_k(G)=1$, as required.

Next, we prove ii). Assume $\delta =\Delta=k+1$.  Clearly, no single vertex will constitute a $k$-forcing set, since each vertex, if it were the only colored vertex, would have more than $k$ non-colored neighbors -- whence the $k$-forcing process would never begin. This means that $F_k(G) \geq 2$. Now, let $\{u,v\}$ be a pair of adjacent vertices in $G$. Color both $u$ and $v$ and no other vertices. Since $u$ has now $k$ non-colored neighbors, and $v$ has now $k$ non-colored neighbors, both $u$ and $v$ force all their neighbors to change color. Now, each of these vertices has at most $k$ non-colored neighbors since they are all adjacent to either $u$ or to $v$. Thus, they in turn will force all of their neighbors to change color on the second step of the $k$-forcing process. Since $G$ is connected, this process will clearly continue until all vertices of $G$ are colored. Therefore, $\{u,v\}$ is a $k$-forcing set and $F_k(G)\leq 2$. This means that $F_k(G)=2$, as required.

Last, we prove iii). Assume that $\Delta \geq k+2$. Let $v$ be a vertex of minimum degree $\delta$. Color $v$ and $\max\{0,\delta -k\}$ of its neighbors, and call this set $S$.  Apply the $k$-forcing process to $G$ using $S$ as the initial set of colored vertices. Since there are at most $k$ uncolored neighbors of $v$, all of the neighbors of $v$ become colored at the first step. Now continue the $k$-forcing process as long as possible. 

Suppose the process does not stop until all of the vertices in $G$ are colored. In this case, $S$ is a $k$-forcing set, so
\[
F_k(G) \leq |S| = |\{v\}| + \max\{0,\delta - k\} = \max\{1,\delta -k +1\}.
\]

On the other hand, suppose the $k$-forcing process stops before all of $G$ is colored. Since $G$ is connected, there is an uncolored vertex $w$ adjacent to a colored vertex $u \neq v$ (If $u=v$, then $w$ would not have been uncolored and the process would not have stopped there, since $v$ already $k$-forced its neighborhood to change color). We now greedily color enough neighbors of $u$ so it can $k$-force $w$ and the process can continue. Let $a(u)$ denote the number of neighbors of $u$ we have to color in order for $u$ to $k$-force its remaining neighbors, including $w$. Note that $a(u) \leq deg(u)-k-1$, since we need at most $deg(u)-k$ neighbors of $u$ colored in order to $k$-force the others, and at least one of the neighbors of $u$ was already colored because it $k$-forced $u$ to change color as $u \neq v$. 
Hence, the proportion of vertices colored by us to total vertices colored, whether by us or by the $k$-forcing process, is:
\[
\frac{a(u)}{a(u)+k} \leq \frac{deg(u)-k-1}{deg(u)-1} \leq \frac{\Delta - (k+1)}{\Delta-1},
\]
where both inequalities comes from monotonicity.

Now, let the process continue as before, and iterate the above steps. Each stop of the process requires coloring more vertices according to the proportion indicated in the upper bound above. We eventually arrive, by greedy construction, at a $k$-forcing set $T$ which satisfies:
\[
|T| \leq \frac{\Delta - (k+1)}{\Delta-1}(n(G)-(\delta+1)) + \max\{1,\delta-k+1\}.
\]
This can be written as,
\[
|T| \leq \frac{(\Delta - k-1)n(G)}{\Delta-1} + \frac{(\Delta-1)\max\{1,\delta-k+1\} - (\Delta-k-1)(\delta+1)}{\Delta-1},
\]
which, after some algebraic simplifications, becomes,
\[
|T| \leq \frac{(\Delta - k-1)n(G)}{\Delta-1} + \frac{\max\{\delta(k+1-\Delta)+k, k(\delta-\Delta +2)\}}{\Delta-1}.
\]
Finally, combining the numerators and recognizing that $F_k(G) \leq |T|$, since $T$ is a $k$-forcing set, the theorem is proven.

\end{proof}

\section{Corollaries of the main result}

In this section, we highlight a few interesting and relevant corollaries to Theorem \ref{main}.  First, by considering the special case of $k=1$, for the zero forcing number, our greedy algorithm produces the following result.

\begin{corollary}\label{cor1}
Let $G=(V,E)$ be a connected graph with maximum degree $\Delta$ and minimum degree $\delta$. Then,
\[
Z(G)=F_1(G) \leq \frac{(\Delta - 2)n(G) - (\Delta -\delta) +2}{\Delta -1}.
\]
\end{corollary}

Next, we can simplify the expression on the right side of the inequality from Theorem \ref{main} by approximating $\max\{\delta(k+1-\Delta)+k, k(\delta-\Delta +2)\}$ as below.

\begin{corollary}\label{cor}
Let $k$ be a positive integer and let $G=(V,E)$ be a connected graph with minimum degree $\delta$ and maximum degree $\Delta \geq k+2$. Then, the following inequality holds, with equality holding only if $G$ is regular of degree $k+2$,
\[
F_k(G) \leq \frac{(\Delta-k-1)n(G) + 2k}{\Delta -1}.
\]
\end{corollary}

In a recent paper of Meyer \cite{Meyer}, it was asked whether there is an upper bound for the zero forcing number of bipartite circulant graphs in terms of the maximum degree $\Delta$ and the order $n(G)$. The corollary above answers this question in the affirmative and in a much more general way, since it is for $k$-forcing and there is no need for the conditions that $G$ be bipartite or that $G$ be a circulant. In particular, by specifying that $k=1$ in the above corollary, we get the corollary below, which is a result previously discovered by Amos, Caro, Davila and Pepper in \cite{ACDP}. 

\begin{corollary}\cite{ACDP}\label{cor3}
Let $G=(V,E)$ be a connected graph with maximum degree $\Delta \geq 2$. Then,
\[
Z(G)=F_1(G) \leq \frac{(\Delta-2)n(G) + 2}{\Delta -1}.
\]
\end{corollary}

The case of equality for the above corollary is left as an open problem in \cite{ACDP}. In light of Corollary \ref{cor1}, we find a necessary condition for equality to hold is that the graph is regular. This sheds some light on the case of equality of Corollary \ref{cor3}.

\section{Comparison to other bounds}

In this short section, we compare Theorem \ref{main} to two theorems from \cite{ACDP}. First, we present the theorems in question.

\begin{theorem}\label{main1}\cite{ACDP}
Let $k$ be a positive integer and let $G=(V,E)$ be a graph on $n(G)\geq 2$ vertices with maximum degree $\Delta\geq k$ and minimum degree $\delta \geq 1$. Then,
\[
F_k(G) \leq \frac{(\Delta-k+1)n(G)}{\Delta -k+1+\min\{\delta,k\}}.
\]
\end{theorem}

\begin{theorem}\label{main2}\cite{ACDP}
Let $k$ be a positive integer and let $G=(V,E)$ be a $k$-connected graph with $n(G)>k$ vertices and $\Delta \ge 2$. Then,
\[
F_k(G) \leq \frac{(\Delta-2)n(G)+2}{\Delta+k-2},
\]
and this inequality is sharp.
\end{theorem}

It is mentioned in \cite{ACDP} that Theorem \ref{main2} improves upon Theorem \ref{main1} for $k$-connected graphs with $\delta \geq k$ and $k \leq 2$, but that Theorem \ref{main1} is superior to Theorem \ref{main2} whenever $k \geq 3$. It turns out that Theorem \ref{main}, and its Corollary \ref{cor}, are at least as good upper bounds as both of these theorems for every $k$. 

More precisely, when $k=1$, the upper bounds from Theorems \ref{main} and \ref{main2} are identical, and both better than Theorem \ref{main1}. On the other hand, when $k\geq 2$, Theorem \ref{main} is superior to both of Theorems \ref{main1} and \ref{main2}, and that result follows essentially from the monotonicity of the function $f(x)=\frac{x}{x+C}$ where $C>0$ is a constant.

\section{Concluding Remarks}
In this article, we have given a greedy algorithm to construct zero forcing and $k$-forcing sets. Corollaries to this construction solve an open problem posed in \cite{Meyer}. We have also given improvements on the non-constructive but similar theorems found in \cite{ACDP}, when $k\geq 2$. We leave as open problems the characterizations of the cases of equality of our results, particularly Corollaries \ref{cor1} and \ref{cor}.


\begin{thebibliography}{widestlabel}

\bibitem{AIM}
AIM Minimum Rank -- Special Graphs Work Group (F. Barioli, W. Barrett, S. Butler, S. Cioaba, D. Cvetkovic, S. Fallat, C. Godsil, W. Haemers, L. Hogben, R. Mikkelson, S. Narayan, O. Pryporova, I. Sciriha, W. So, D. Stevanovic, H. van der Holst, K.V. Meulen, A. W. Wehe), Zero forcing sets and the minimum rank of graphs, \emph{Linear Algebra and its Applications}, Volume 428, Issue 7, 1 April 2008, Pages 1628-1648.


\bibitem{ACDP}
David Amos, Yair Caro, Randy Davila, Ryan Pepper, Upper bounds on the k-forcing number of a graph, submitted and under review, arXiv:1401.6206v1 [math.CO] 23 Jan 2014.


\bibitem{Burgarth}
Daniel Burgarth and Vittorio Giovannetti, Full control by locally induced relaxation, \emph{Phys. Rev. Lett.} 99, 100501 (2007).	
	




\bibitem{Chilakammari}
Kiran B. Chilakammari, Nathaniel Dean, Cong X. Kang, Eunjeong Yi, Iteration Index of a Zero Forcing Set in a Graph, \emph{Bull. Inst. Combin. Appl.} Vol. 64 (2012) pp. 57-72.




\bibitem{Dreyer}
Paul A. Dreyer Jr. and Fred S. Roberts, Irreversible $k$-threshold processes: Graph-theoretic threshold models of the spread of disease and of opinion, emph{Discrete Applied Mathematics} 157 (2009), 1615-1627.

\bibitem{Edholm}
C.J. Edholm, L. Hogben, M. Huynh, J. LaGrange, and D.D. Row, Vertex and edge spread of the zero forcing number, maximum nullity, and minimum rank of a graph, \emph{Linear Algebra and its Applications}, 436 (2012), 4352-4372.
	
	
\bibitem{Eroh}
Linda Eroh, Cong Kang, Eunjeong Yi, Metric dimension and zero forcing number of two families of line graphs
	

   


  
	
\bibitem{Hogben}
Leslie Hogben, My Huynh, Nicole Kingsley, Sarah Meyer, Shanise Walker, Michael Young, Propagation time for zero forcing on a graph, \emph{Discrete Applied Mathematics} 160 (2012), 1994-2005.

\bibitem{Meyer}
Seth A. Meyer, Zero forcing sets and bipartite circulants, \emph{Linear Algebra and its Applications} 436 (2012), 888-900.

	
\bibitem{Row2}
Darren Row, A technique for computing the zero forcing number of a graph with a cut-vertex, \emph{Linear Algebra and its Applications} 436 (2012), 4423-4432.

\bibitem{Row}
Darren Row, \emph{Zero forcing number: Results for computation and comparison with other graph parameters}, Ph.D. Thesis, Iowa State University, 2011.
	    
\bibitem{West}
Douglas B. West, \emph{Introduction to Graph Theory}, second ed., Prentice Hall, Upper Saddle River, NJ, 2001.
	
\bibitem{Yi}
Eunjeong Yi, On Zero Forcing Number of Permutation Graphs, \emph{Combinatorial Optimization and Applications}, Lecture Notes in Computer Science Volume 7402, 2012, pp 61-72.

\end{thebibliography}
\end{document}